\documentclass[sigconf,screen=true,natbib=true]{acmart}

\usepackage{amsmath}
\usepackage{amsthm}
\numberwithin{equation}{section}

\usepackage{amssymb}
\usepackage{amsfonts}
\usepackage[utf8]{inputenc}
\usepackage{enumerate}
\usepackage{hyperref}
\usepackage{natbib}
\usepackage{paralist}
\usepackage{url}
\usepackage{hyperref}
\usepackage{color}
\usepackage{framed}
\usepackage{mathrsfs}
\usepackage{graphicx}
\usepackage{bm}
\usepackage{lipsum}                     % Dummytext
\usepackage{xargs}                      % Use more than one optional parameter in a new commands
\usepackage{xcolor}  % Coloured text etc.
\usepackage{float}
\usepackage[all,pdf]{xy}
\usepackage{mathtools}
\usepackage{mathrsfs}
\usepackage{diagbox}

\renewcommand\theenumi{\roman{enumi}}

\usepackage{graphicx}
\title{A Characterization of Perfect Strategies for Mirror Games}

\author{Sizhuo Yan}
\affiliation{
\institution{Key Lab of Mathematics Mechanization, AMSS }
\department{ University of Chinese Academy of Sciences}
\city{Beijing, 100190}
\country{China}
}
\email{yansizhuo@amss.ac.cn}

\author{Jianting Yang}
\affiliation{
\institution{Key Lab of Mathematics Mechanization, AMSS }
\department{ University of Chinese Academy of Sciences}
\city{Beijing, 100190}
\country{China}
}
\email{yangjianting@amss.ac.cn}

\author{Tianshi Yu}
\affiliation{
	\institution{Key Lab of Mathematics Mechanization, AMSS }
	\department{ University of Chinese Academy of Sciences}
	\city{Beijing, 100190}
	\country{China}
}
\email{yutianshi@amss.ac.cn}

\author{Lihong Zhi}
\affiliation{
\institution{Key Lab of Mathematics Mechanization, AMSS }
\department{ University of Chinese Academy of Sciences}
\city{Beijing, 100190}
\country{China}
}
\email{lzhi@mmrc.iss.ac.cn}


\begin{document}

\begin{abstract}
	We associate mirror games with the universal game algebra and use the *-representation to describe  quantum commuting operator strategies.  We provide an algebraic characterization of  whether or not a mirror game has perfect commuting operator strategies. This new characterization  uses  a smaller algebra introduced by Paulsen and others for synchronous games and   the noncommutative Nullstellensatz developed by Cimpric, Helton and collaborators. An algorithm based on noncommutative Gr\"obner basis computation and semidefinite programming is given for certifying  that a given mirror game has no perfect commuting operator strategies.

\end{abstract}
\maketitle

\newtheorem{defi}{Definition}[section]
\newtheorem{thm}{Theorem}[section]
\newtheorem{cor}[thm]{Corollary}
\newtheorem{prop}[thm]{Proposition}
\newtheorem{lem}{Lemma}[section]
\newtheorem{exa}{Example}[section]
\section{Introduction}
Quantum nonlocal games have been  an active area of research for mathematicians, physicists, and computer scientists in past decades.
The violation of Bell inequality has verified the non-locality of quantum mechanics \cite{bell1964einstein}, which  can   be explained in the framework of nonlocal games \cite{cleve2004consequences,palazuelos2016survey}.
A nonlocal game  has two or multiple  players and a verifier. The verifier sends a question to each player separately, and each player sends an answer back to the verifier without communicating with the others.
The verifier  determines whether  the players win for the given questions and answers.
We have a classical strategy if the players can only share classical information.
We have a quantum strategy if we allow the players to share quantum information.
Bell inequality violations have been proved in the CHSH game \cite{clauser1969proposed}, where the winning probability using classical strategies is at most $3/4$, while a quantum strategy using an entangled state shared by two players can achieve a success probability  $\cos^2(\pi/8) \approx 0.85$.
Noncommutative  Positivstellens\"atze  have been used to study nonlocal games in \cite{navascues2008convergent,doherty2008quantum}.

A synchronous game is a nonlocal game with two players called Alice and Bob, where  Alice and Bob are sent the same question and win if and only if they send the same response.  
Paulsen and his collaborators found a simpler  formulation using  a smaller algebra and hard zeroes to study synchronous games in \cite{paulsen2016estimating,helton2019algebras}.  It has been shown
that the  success probability   of a synchronous game is given by the trace of a bilinear function on a smaller algebra, see Theorem 5.5 in  \cite{paulsen2016estimating}, and Theorem 3.2 in \cite{helton2019algebras}.
In \cite{helton2019algebras,watts2021noncommutative,watts20203xor}, they  give algebraic characterizations of perfect quantum commuting operator strategies for a general game using  noncommutative Nullstellens\"atze \cite{cimpric2015real,cimprivc2013noncommutative,cimprivc2014real}
and Positivstellens\"atze \cite{MC2001,H2002,HMS2004,BKP2016}.
Theorem 8.3 and 8.7 in \cite{watts2021noncommutative}  provide a simplified version of the Nullstellensatz theorem for synchronous games.

In \cite{lupini2020perfect}, Lupini, etc. introduce a new class of nonlocal games called imitation games, in which another player's answer completely determines each player's answer.
Any  synchronous game is an imitation game as the players send the same answers for the same questions.
Some imitation games are not synchronous, such as
mirror games, unique games \cite{rao2011parallel}, and variable assignment games \cite{lupini2020perfect}. Lupini, etc.,  associates a C*-algebra with
any imitation game and characterizes  perfect quantum commuting strategies  in terms of the properties of this C*-algebra.

As an interesting  subclass of imitation games, mirror games   include unique games and synchronous games.  Theorem 5.5 in \cite{paulsen2016estimating} for synchronous games has been generalized  to Theorem 6.1  in  \cite{lupini2020perfect} for  mirror games, and a  representation of  perfect quantum commuting strategies  for mirror games in terms of traces is also  given in the paper.
It is natural to ask whether one can  obtain  similar results as  Theorem 8.3 and  8.7 in \cite{watts2021noncommutative} for mirror games. We answer the question in  Theorem \ref{thm3.2}: we provide  an algebraic  characterization of  whether or not a mirror game has perfect commuting operator strategies   based on a noncommutative Nullstellensatz and sums of squares.
This new characterization  uses  a smaller algebra introduced by Paulsen and others for synchronous games and   the noncommutative Nullstellensatz developed by Cimpric, Helton, and  collaborators \cite{cimpric2015real,cimprivc2013noncommutative,cimprivc2014real}.
An example is given to demonstrate  how to use noncommutative  Gr\"obner basis algorithm \cite{mora1986grobner} and semidefinite programming \cite{vandenberghe1996semidefinite} to verify that a given mirror game has  no perfect commuting operator strategies.
It would be interesting to see how to extend these results  to imitation games.

The paper is organized as follows. Section \ref{sec2} introduces some preliminary  results and definitions of nonlocal games. Some background material on classical strategies and quantum strategies of nonlocal games  are included. We also introduce the  universal game algebra and its *-representation.  Section  \ref{sec3} contains our main result on characterizing   whether or not a mirror game has  perfect commuting operator strategies   based on a noncommutative Nullstellensatz and sums of squares. Finally, Section \ref{sec4} shows how to  use  noncommutative  Gr\"obner basis and semidefinite programming to verify that a given mirror game has  no perfect commuting operator strategies. A running example is given  to demonstrate the computations.

\section{Preliminaries}\label{sec2}

A nonlocal game $ \mathcal{G} $ involves a verifier and two players, Alice and Bob. For fixed  non-empty finite sets $ X,Y $ and $ A,B $, there exists a distribution $\mu $ on $ X\times Y $.
After choosing a pair $ (x,y)\in X\times Y $ randomly according to $ \mu(x,y)$, the verifier sends elements $ x $ to Alice and $ y $ to Bob as questions. Alice and Bob  send the verifier corresponding answers $ a\in A $ and $ b\in B $.
After receiving an answer from each player, the verifier  evaluates the  scoring function
\begin{equation}
\lambda:~X\times Y\times A\times B\longrightarrow \{0,1\}
\end{equation}
If $ \lambda(x,y, a,b)=1$, we say Alice and Bob win; otherwise, they lose the game.
Alice and Bob know the sets $ X, Y, A, B $ and the scoring function $\lambda$,
but they can't communicate during the game.  Alice and Bob can make some arrangements before the game starts.

A deterministic strategy for the players consists of two functions:
\begin{equation}
a:X\longrightarrow A, \, \,  b:Y\longrightarrow B,
\end{equation}
and Alice sends $ a(x) $ to the verifier if she receives $x$, and Bob sends $ b(y) $ to the verifier if he receives $y$.
Given a deterministic strategy, the players win the game $ \mathcal{G} $ with an expectation
\begin{equation}
\sum_{x,y}\mu(x,y)\lambda(x,y,a(x),b(y)).
\end{equation}

We can also give a probabilistic strategy for $ \mathcal{G} $ as follows: for each pair $ (x,y)\in X\times Y$, let Alice and Bob have mutually independent distributions $ p_{x, a}, q_{y,b}$ for $a\in A, b\in B$. When the players receive the questions $(x,y)$, Alice sends the answer $ a $ to the verifier with probability $ p_{x, a} $ and Bob sends the answer $ b $ to the verifier with probability $ q_{y,b}$.
The winning expectation is
\begin{equation}
\sum_{x,y,a,b}\mu(x,y)p_{x,a}q_{y,b}\lambda(x,y,a,b).
\end{equation}

All deterministic strategies and probabilistic strategies are collectively referred to as classical strategies. We record the set of all classical strategies as $ C_c $,  which is  a closed set. Notice that any probabilistic strategy can be expressed as a convex combination of deterministic strategies so that the maximal winning expectation of a nonlocal  game  $\mathcal{G} $  with classical strategies is always obtained by some deterministic strategy. The  classical value of $\mathcal{G} $ is defined as  the maximal winning expectation
\begin{equation}
\omega_c(\mathcal{G})=\max_{a,b}\sum_{x,y}\mu(x,y)\lambda(x,y,a(x),b(y)).
\end{equation}

We use the Dirac notation in quantum information to represent the unit vector (a state) in Hilbert space. 
If  Alice and Bob are allowed to  share a quantum  entangled state $ |\psi\rangle\in\mathcal{H}_{A}\otimes\mathcal{H}_{B} $, where both $\mathcal{H}_{A}$ and $\mathcal{H}_{B}$ are finite-dimensional Hilbert space, and then they can have a quantum strategy described as follows:
\begin{itemize}
	\item If Alice receives $ x $, she performs the projection-valued measure
	(PVM)  $P_{x,a} $ on  $ \mathcal{H}_{A} $ part of  $ |\psi\rangle $ and sends the measurement result $ a $  to the verifier.
	\item If Bob receives $ y $, he performs the PVM   $Q_{y,b}$ on $ \mathcal{H}_{B} $ part of $ |\psi\rangle $ and sends the measurement result $ b $ to the verifier.
\end{itemize}
If we replace PVM by POVM (positive operator-valued measure), the results below will also hold \cite{paulsen2015quantum,fritz2014operator}.

We record the set of all finite-dimensional quantum strategies as $ C_q $.
If we drop the requirement of finite dimension, i.e., $\mathcal{H}_{A},\mathcal{H}_{B} $  can be  infinite-dimensional Hilbert spaces, then we get a set of  quantum strategies denoted as $ C_{qs} $. Slofstra  \cite{slofstra2019set,scholz2008tsirelson} has proved that neither $ C_q $ nor $ C_{qs} $ is a closed set. We denote the closure of $ C_q$ as $C_{qa}$.
It is evident that
\[
C_c\subseteq C_q\subseteq C_{qs}\subseteq C_{qa}
\]
Each of the above $"\subseteq"$ is strictly inclusive. The first strict inclusion comes from Bell's inequality, and the last two strict inclusions come from results in  \cite{slofstra2019set,scholz2008tsirelson, dykema2019non, coladangelo2020inherently}.

The winning expectation for the given  quantum strategy is
\begin{equation} \sum_{x,y,a,b}\mu(x,y)\cdot\langle\psi|P_{x,a}\otimes Q_{y,b}|\psi\rangle\cdot\lambda(x,y,a,b).
\end{equation}
If we take all of the quantum strategies, the supremum of winning expectations is
\begin{equation}
\omega_q(\mathcal{G})=\underset{\tiny{\begin{array}{c}\mathcal{H}_A,\mathcal{H}_B,\psi,\\P_{x,a},Q_{y,b}\end{array}}}{\sup}\sum_{x,y,a,b}\mu(x,y)\cdot\langle\psi|P_{x,a}\otimes Q_{y,b}|\psi\rangle\cdot\lambda(x,y,a,b),
\end{equation}
which is called the quantum value of $ \mathcal{G} $. The quantum value can certainly be attained  in $ C_{qa} $, but not necessarily in $ C_q $ or $ C_{qs} $.

Now we give a quantum commuting operator strategy for $ \mathcal{G} $ as follows. Let $ \mathcal{H} $ be a (perhaps infinite-dimensional) Hilbert space, $ |\psi\rangle\in \mathcal{H} $, and for every $ (x,y)\in X\times Y $, Alice and Bob have PVMs $ \{E(1)^{x}_{a},~a\in A\} $ and $ \{E(2)^{y}_{b},~b\in B\} $, respectively. Those two sets of PVMs satisfy the following conditions:
\begin{equation}
E(1)^{x}_{a}E(2)^{y}_{b}=E(2)^{y}_{b}E(1)^{x}_{a},~~~\forall~(x,y,a,b)\in X\times Y\times A\times B.
\end{equation}
When Alice receives an input $x$, she performs
$ \{E(1)^{x}_{a},~a\in A\} $ on $ |\psi\rangle $ and sends the result $ a $ to the verifier;
Similarly, when Bob receives an input $ y $, he performs
$ \{E(2)^{y}_{b},~b\in B\} $ on $ |\psi\rangle $ and sends the result $ b $ to the verifier.

We denote the set of all the quantum commuting operator strategies as $ C_{qc} $. We know that $ C_{qc} $ is closed \cite{fritz2012tsirelson}. Given a quantum commuting operator strategy of $ \mathcal{G} $, the winning expectation is
\begin{equation}
\sum_{x,y,a,b}\mu(x,y)\cdot\langle\psi|E(1)^{x}_{a}\cdot E(2)^{y}_{b}|\psi\rangle\cdot\lambda(x,y,a,b).
\end{equation}
Then the supremum of winning expectation (note that it can certainly be obtained) is
\begin{equation}
\omega_{co}(\mathcal{G})=\underset{\tiny{\begin{array}{c}\mathcal{H},\psi,\\E(1)^{x}_{a},E(2)^{y}_{b}\end{array}}}{\sup}
\sum_{x,y,a,b}\mu(x,y)\cdot\langle\psi|E(1)^{x}_{a}\cdot E(2)^{y}_{b}|\psi\rangle\cdot\lambda(x,y,a,b)
\end{equation}
which is called the quantum commuting operator value of $\mathcal{G}$.

It is easy to see that $ C_{qa}\subseteq C_{qc}$ \cite{fritz2012tsirelson}, so that we have $ \omega_c(\mathcal{G})\leq\omega_q(\mathcal{G})\leq\omega_{co}(\mathcal{G}) $.
If we restrict the Hilbert space $ \mathcal{H} $ to be  finite-dimensional in the commuting operator strategies, then $ \omega_q(\mathcal{G})=\omega_{co}(\mathcal{G}) $ (see \cite{tsirelsonbell, scholz2008tsirelson}).
There exist games $ \mathcal{G} $ for which $ \omega_q(\mathcal{G})<\omega_{co}(\mathcal{G}) $ in the infinite-dimensional case, see \cite{slofstra2019set,fritz2012tsirelson}.
The problem of whether $ C_{qc}=C_{qa} $ is the famous Tsirelson's problem, and it is true if and only if the Connes' embedding conjecture is true \cite{connes1976classification}. Kirchberg  shows that Connes’ conjecture has several equivalent reformulations in operator algebras and Banach space theory \cite{kirchberg1993non}.
In \cite{klep2008connes}, Klep and Schweighofer show that Connes' embedding conjecture on von Neumann algebras is equivalent to the tracial version of the  Positivstellensatz.
In 2020, Ji and his collaborators  proved $ MIP^*=RE $, which implies  that Connes' embedding conjecture  is false \cite{ji2020mip}.  But we still don't know an explicit  counterexample.   See \cite{goldbring2022connes,ozawa2013connes,fritz2012tsirelson} for recent results on the
Connes' embedding problem.  This is the main motivation for us to study quantum nonlocal games.

We say a strategy is perfect if and only if the players can certainly win the game with this strategy. A natural problem is to ask whether there exists a perfect strategy in $ C_c $ (or $ C_q, C_{qs}, C_{qa}, C_{qc} $) for a given game $ \mathcal{G} $.

In \cite{lupini2020perfect}, the authors introduce a new class of nonlocal games called imitation games, and they provide an algebraic characterization of  perfect commuting operator strategies for these games. In this paper, we mainly discuss the mirror game, which is a  special subclass of imitation games.

\begin{defi}[mirror game]
	Let $ \mathcal{G} $ be a nonlocal game with a question set $ X\times Y $, an answer set $ A\times B $ and a scoring function $ \lambda:X\times Y\times A\times B\rightarrow\{0,1\} $. The distribution on $ X\times Y $ is the uniform distribution. We say $ G $ is a mirror game if there exist functions $ \xi:X\rightarrow Y $ and $ \eta:Y\rightarrow X $ such that:
	\begin{eqnarray}
	\lambda(x, \xi(x), a, b) \lambda\left(x, \xi(x), a^{\prime}, b\right)=0, ~\forall x \in X, a\neq a^{\prime} \in A, b \in B,\\
	\lambda(\eta(y), y, a, b) \lambda\left(\eta(y), y, a, b^{\prime}\right)=0, ~\forall y \in Y, a \in A, b \neq b^{\prime} \in B.
	\end{eqnarray}
%\begin{eqnarray}
%	\lambda(x, \xi(x), a, b) \lambda\left(x, \xi(x), a^{\prime}, b\right)=0,  x \in X, a, a^{\prime} \in A, b \in B, a \neq a^{\prime},\\
%	\lambda(\eta(y), y, a, b) \lambda\left(\eta(y), y, a, b^{\prime}\right)=0,  y \in Y, a \in A, b, b^{\prime} \in B, b \neq b^{\prime} .
%	\end{eqnarray}
\end{defi}

%\begin{exa}\label{example1}
%\begin{exa}  

\begin{example}\label{example1}
%\end{exa}[continues = ex:running example-1]\label{example1}
%\begin{example}[continues = ex:running example-1]\label{ex:running example-2}
	Let  $ X=Y=A=B=\{0,1\} $ and the scoring function $ \lambda $ be given  as follows:
	
	\begin{center}
		\begin{tabular}{|c|c|c|c|c|}\hline
			\diagbox{$(a,b)$}{$ \lambda $}{$(x,y)$} & (0,0) & (0,1) & (1,0) & (1,1)  \\\hline
			(0,0)                                                     & 1     & 0     & 1     & 0      \\\hline
			(0,1)                                                     & 0     & 0     & 1     & 1      \\\hline
			(1,0)                                                     & 0     & 1     & 0     & 0      \\\hline
			(1,1)                                                     & 1     & 0     & 0     & 1     \\\hline
		\end{tabular}
	\end{center}
	We can  check that $ \mathcal{G}=(X,Y,A,B,\lambda) $ is a mirror game with
	\[\xi:0\mapsto 0,~1\mapsto 0, \, \eta:0\mapsto 0,~1\mapsto 1.\]
%\end{exa}
\end{example}

%\begin{example}[continues = ex:running example-1]\label{ex:running example-2}
%\end{example}

%\end{example}

%\begin{exa}[continues = ex:running example-1]\label{ex:running example-2}
%let sss 
%\end{exa}

%At present the method in this paper can't be extended to imitation games, since we don't know how %to algebraize the proof of Theorem 5.1 in the paper \cite{lupini2020perfect}.In the future we %would continue working on this field.

%In order to give a characterization of  whether a mirror game has a perfect strategy in $ C_{qc} $,
We use the  universal game algebra and representation defined in \cite{watts2021noncommutative} to describe the relations between the PVMs in the commuting operator strategy below.

\begin{defi}[universal game algebra]
	Let
\begin{equation}
\mathbf{e}=(e(1)^x_a)_{x\in X,a\in A}\cup(e(2)^y_b)_{y\in Y,b\in B}
 \end{equation}
 and $ \mathbb{C}\langle\mathbf{e}\rangle $ be the noncommutative free algebra generated by the tuple $ \mathbf{e} $. Let $ \mathscr{I} $ be the two-sided ideal generated by the following polynomials:
	\begin{equation}
	\begin{aligned}
	&\quad\left\{e(1)^{x}_{a}e(2)^{y}_{b}-e(2)^{y}_{b}e(1)^{x}_{a}\mid\forall x,y,a,b\right\}\\&\cup\left\{(e(1)^{x}_{a})^2-e(1)^{x}_{a}\mid\forall x,a\right\}\cup\left\{(e(2)^{y}_{b})^2-e(2)^{y}_{b}\mid\forall y,b\right\}
	\\&\cup\left\{e(1)^{x}_{a_1}e(1)^{x}_{a_2}\mid\forall x,~a_1\neq a_2\right\}\cup\left\{e(2)^{y}_{b_1}e(2)^{y}_{b_2}\mid \forall y,~b_1\neq b_2\right\}\\&\cup\left\{\sum_{a\in A}e(1)^{x}_{a}-1\mid \forall x\right\}\cup\left\{\sum_{b\in B}e(2)^{y}_{b}-1\mid \forall y\right\}.
	\end{aligned}
	\end{equation}
	Then we define $ \mathcal{U}=\mathbb{C}\langle\mathbf{e}\rangle/\mathscr{I} $ and equip $ \mathcal{U} $ with the involution induced by
	\begin{equation}
	(e(1)^{x}_{a})^*=e(1)^{x}_{a},~(e(2)^{y}_{b})^*=e(2)^{y}_{b}.
	\end{equation}
	where the $ "*" $ of a complex number is its conjugate. We call $ \mathcal{U} $  the universal game algebra of $ \mathcal{G} $.
\end{defi}

For the universal game algebra  $ \mathcal{U} $,  we can use *-representation to describe a commuting operator strategy.
A *-representation of $ \mathcal{U} $ is a  unital *-homomorphism
\begin{equation}
 \pi:\mathcal{U}\rightarrow\mathcal{B}(\mathcal{H}),
 \end{equation}
where $ \mathcal{B}(\mathcal{H}) $ denotes the set of bounded linear operators on a Hilbert space $ \mathcal{H} $ and $ \pi $ satisfies $ \pi(u^*)=\pi(u)^*, \forall u\in U $.
It is obvious that any commutative PVMs
$ \{E(1)^{x}_{a},~a\in A\} $ and $ \{E(2)^{y}_{b},~b\in B\} $ can be obtained by the unital *-homomorphism
\begin{equation}
 \pi:e(1)^{x}_{a}\mapsto E(1)^{x}_{a},~e(2)^{y}_{b}\mapsto E(2)^{y}_{b},
 \end{equation}
and given an arbitrary unital *-homomorphism, the image of $\mathcal{U}$'s generators is commutative PVMs.
Therefore,  each commuting operator strategy corresponds to a pair $ (\pi,|\psi\rangle) $, where
$\pi:\mathcal{U}\rightarrow\mathcal{B}(\mathcal{H})$
is a *-representation and $ |\psi\rangle\in\mathcal{H} $ is a state  (a unit vector).
We can use the language of representation to rewrite $ \omega_{co}(\mathcal{G}) $ as follows:
\begin{equation}
\omega_{co}(\mathcal{G})=\underset{\pi,\psi}{\sup}\langle\psi|\pi(\Phi_{\mathcal{G}})|\psi\rangle,
\end{equation}
where
\begin{equation}
\Phi_{\mathcal{G}}=\sum_{x,y}\sum_{a,b}\mu(x,y)\lambda(x,y,a,b)e(1)^{x}_{a}e(2)^{y}_{b},
\end{equation}
and the supremum is taken over all *-representations $ \pi $ of $ \mathcal{U} $ into bounded operators on a Hilbert space $ \mathcal{H} $ and state $ |\psi\rangle\in\mathcal{H} $.

Since we assume that  $ \mu $ is a uniform distribution, $ \Phi_{\mathcal{G}} $ can be simplified to
\begin{equation}
\Phi_{\mathcal{G}}=\dfrac{1}{|X|\cdot|Y|}\sum_{x,y}\sum_{a,b}\lambda(x,y,a,b)e(1)^{x}_{a}e(2)^{y}_{b}.
\end{equation}
It's obvious that a game $ \mathcal{G} $ has a perfect commuting operator strategy if and only if
\begin{equation}
\omega_{co}(\mathcal{G})=1.
\end{equation}

We also need the concept  of tracial linear functional and tracial state.
\begin{defi}
	A linear mapping $ \tau:\mathcal{A}\rightarrow\mathbb{C} $ on an algebra $ \mathcal{A} $ is said to be tracial if and only if
	\begin{equation}
	\tau(ab)=\tau(ba),~\forall a,b\in\mathcal{A}.
	\end{equation}
	Given a Hilbert space $ \mathcal{H} $ and an operator algebra  $\mathcal{A}$ acting on $ \mathcal{H} $, a state $ |\psi\rangle\in\mathcal{H} $  is called a tracial state 
	if the linear mapping it induces is tracial, i.e.
	\begin{equation}
	\langle\psi|ab|\psi\rangle=\langle\psi|ba|\psi\rangle,~\forall a,b\in\mathcal{A}.
	\end{equation}
	Especially if $ \mathcal{A} $ is a von Neumann algebra, and there exists such a tracial linear mapping $ \tau $ on $ \mathcal{A} $, we say $ (\mathcal{A},\tau) $ is a tracial von Neumann algebra.
\end{defi}
The definition of determining set is given in \cite{watts2021noncommutative}.
\begin{defi}[determining set]\label{def2.3}
	Let $ \mathcal{G} $ be a nonlocal game; its universal game algebra is $ \mathcal{U} $. A set $ \mathcal{F}\subseteq\mathcal{U} $ is denoted as a determining set of $ \mathcal{G} $ if it satisfies that a pair $ (\pi,|\psi\rangle) $ is a perfect commuting operator strategy if and only if 
	$ \pi(\mathcal{F})|\psi\rangle=\{0\} $.
\end{defi}

According to Theorem 3.5 in \cite{watts2021noncommutative}, given any nonlocal   game, we have a natural determining set:

\begin{prop}\label{thm2.1}
	Let $ \mathcal{G}=(X,Y,A,B,\lambda) $ be a nonlocal game, the set of invalid elements
	\begin{equation}
	\mathcal{N}=\{e(1)^{x}_{a}e(2)^{y}_{b}\mid \lambda(x,y,a,b)=0\}
	\end{equation}
	is a determining set. We call it the invalid determining set.
\end{prop}
%\begin{proof}
%By Theorem 3.5 in \cite{watts2021noncommutative}.
%\end{proof}

\begin{cor}
	The left ideal $ \mathcal{L}(\mathcal{N}) $ generated by $ \mathcal{N} $ is also a determining set.
\end{cor}

For a mirror game $ \mathcal{G} $, suppose its universal game algebra is $ \mathcal{U} $, and we define:
\begin{eqnarray}
f^{\eta(y)}_{y,b}&=\sum_{a\in A,\lambda(\eta(y),y,a,b)=1}e(1)^{\eta(y)}_{a},\\
g^{\xi(x)}_{x,a}&=\sum_{b\in B,\lambda(x,\xi(x),a,b)=1}e(2)^{\xi(x)}_{b}.
\end{eqnarray}
\begin{defi}
	Let $\mathcal{G}=(X, Y, A, B, \lambda)$ be a nonlocal game. For $x \in X, y \in Y, a \in A$ and $b \in B$, denote
	\begin{eqnarray}
	& E_{x, y}^a=\{b \in B: \lambda(x, y, a, b)=1\}\\
	&E_{x, y}^b=\{a \in A: \lambda(x, y, a, b)=1\}.
	\end{eqnarray}
	We define a mirror game as regular if and only if
	\begin{equation}	
	\cup_{a \in A} E_{x, \xi(x)}^a=B~\text{and}~\cup_{b \in B} E_{\eta(y), y}^b=A, \,  \forall x \in X, y \in Y.\end{equation}
	Remark that this condition appeared in \cite{lupini2020perfect} firstly, but they didn't name it.
\end{defi}
\begin{lem}\label{lem2.1}
	A mirror game $ \mathcal{G} $ is regular if and only if the universal game algebra satisfies:
	\begin{equation}\label{reg}
	\sum_{a\in A}g^{\xi(x)}_{x,a}=1, ~\, \forall x\in X~\text{and}~
	\sum_{b\in B}f^{\eta(y)}_{y,b}=1, ~ \, \forall y\in Y.
	\end{equation}
\end{lem}
\begin{proof}
	By the definition of regularity  and the universal game algebra.
\end{proof}

%\begin{exa}[Example \ref{example1} continued] \label{example2}
 {Example} \ref{example1} (continued).
{\it
	For the mirror game $ \mathcal{G} $ defined in Example \ref{example1},
%	Let $ \mathbb{C}\langle e(1)\rangle $ be the free algebra generated by
%	$ \{e(1)^{i}_{j}\mid(i,j)\in\{0,1\}^2\} $, and  $ \mathcal{U}(1)$ be the subalgebra of the universal game algebra $ \mathcal{U} $ generated by $\{e(1)^{i}_{j}\mid(i,j)\in\{0,1\}^2\}$.
%	
%	By values in  the above table, 
	we can compute that
	$$
	f^{\eta(0)}_{0,0}=e(1)^{0}_{0},~f^{\eta(0)}_{0,1}=e(1)^{0}_{1},~f^{\eta(1)}_{1,0}=0,~f^{\eta(1)}_{1,1}=e(1)^{0}_{0}+e(1)^{0}_{1}=1.
	$$
It is easy to check that 
 $\sum_{b\in B}f^{\eta(y)}_{y,b}=1, ~ \, \forall y\in Y.$
	 Similarly, we have
	$$
	g^{\xi(0)}_{0,0}=e(2)^{0}_{0},~g^{\xi(0)}_{0,1}=e(2)^{0}_{1},~g^{\xi(1)}_{1,0}=1,~g^{\eta(1)}_{1,1}=0.
	$$
 It is true that    $\sum_{a\in A}g^{\xi(x)}_{x,a}=1, ~\, \forall x\in X.$
	Hence,    $ \mathcal{G} $ is a regular mirror game. 
} %\end{exa}

In the following sections, we'll only consider   regular mirror games.

\section{Main Result}\label{sec3}

Given a universal game algebra $ \mathcal{U} $ of
a  nonlocal game $\mathcal{G}$, a general noncommutative
Nullstellensatz developed by Cimpric, Helton, and their collaborators \cite{cimprivc2013noncommutative,cimprivc2014real} has been adapted to  Theorem 4.1 and 4.3 in \cite{watts2021noncommutative} to show that  $\mathcal{G}$ has a perfect commuting operator strategy if and only if there exists a *-representation $\pi: \mathcal{U} \rightarrow \mathcal{B}(\mathcal{H})$ and a state $|\psi\rangle \in \mathcal{H}$ satisfying
\begin{equation}
\pi(\mathcal{L}(\mathcal{N}))|\psi\rangle=\{0\},
\end{equation}
which is also equivalent to
\begin{equation}
-1 \notin \mathcal{L}(\mathcal{N})+\mathcal{L}(\mathcal{N})^*+\operatorname{SOS}_{\mathcal{U}},
\end{equation}
where
\begin{equation}
\operatorname{SOS}_{\mathcal{U}}=\left\{\sum_{i=1}^{n}u_i^{*}u_i\mid u_i\in \mathcal{U},~n\in\mathbb{N}\right\},
\end{equation}
$\mathcal{L}(\mathcal{N})$ is the left ideal generated by the invalid determining set  $\mathcal{N}$.

For synchronous games, the authors use a smaller algebra $ \mathcal{U}(1) $  which is the subalgebra of $\mathcal{U}$ generated by $ e(1)^{x}_{a} $,
and prove that a synchronous game has a perfect commuting operator strategy if and only if there exists a
*-representation $\pi': \mathcal{U}(1) \rightarrow \mathcal{B}(\mathcal{H})$ and a tracial state $|\psi\rangle \in \mathcal{H}$ satisfying
\begin{equation}
\pi'(\mathcal{J}(\rm{synch} \mathcal{B}(1))) |\psi\rangle=\{0\},
\end{equation}
where
$\mathcal{J}(\rm{synch} \mathcal{B}(1))$ is a two-sided ideal in $\mathcal{U}(1)$,
see  Theorem 8.3 and  8.7 in \cite{watts2021noncommutative}.

In  Theorem \ref{thm3.2}, we  generalize Theorem 8.3 and  8.7 in \cite{watts2021noncommutative} for mirror games and
provide  a characterization of  whether or not a mirror game has perfect commuting operator strategies  using   smaller algebras $\mathcal{U}(1)$ and $ \mathcal{U}(2) $, where    $\mathcal{U}(1) $ is the subalgebra of $ \mathcal{U} $ generated by $ e(1)^{x}_{a} $ only, and $ \mathcal{U}(2) $ is  the subalgebra of $ \mathcal{U} $ generated by $ e(2)^{y}_{b} $ only.

Let  $ \mathcal{J}(\operatorname{mir1}) $ be  the two-sided ideal of $ \mathcal{U}(1) $ generated by
\begin{equation}
\left\{e(1)^{x}_{a}f^{\eta(y)}_{y,b}\mid\lambda(x,y,a,b)=0\right\},
\end{equation}
and $ \mathcal{J}(\operatorname{mir2}) $ be  the two-sided ideal of $ \mathcal{U}(2) $ generated by
\begin{equation}
\left\{e(2)^{y}_{b}g^{\xi(x)}_{x,a}\mid\lambda(x,y,a,b)=0\right\}.
\end{equation}

%\begin{exa}\label{example3}
Example \ref{example1} (continued). {\it{
	Let's continue the computation in Example \ref{example1}.
	The two-sided ideal $ \mathcal{J}(\operatorname{mir1}) $  is generated by the following elements:
	$$
	\{e(1)^{0}_{0}e(1)^{0}_{1},~e(1)^{0}_{1}e(1)^{0}_{0},~0,~e(1)^{0}_{0},~e(1)^{0}_{1},~e(1)^{1}_{1}e(1)^{0}_{0},~e(1)^{1}_{1}e(1)^{0}_{1},0,0\}.
	$$
	
	It is clear that
	$ \mathcal{J}(\operatorname{mir1}) $ is generated by
	$\{e(1)^{0}_{0},~e(1)^{0}_{1}\} ~{\text{in}}~  \mathcal{U}(1).$}}
%\end{exa}

\begin{thm}[main result]\label{thm3.2}
	A regular mirror game with its  universal game algebra $ \mathcal{U} $ and invalid determining set $ \mathcal{N} $ has a perfect commuting operator strategy  if and only if any of the equivalent conditions are satisfied:
	\begin{enumerate}[(1)]
		\item There exists a *-representation $\pi: \mathcal{U} \rightarrow \mathcal{B}(\mathcal{H})$ and a state $|\psi\rangle \in \mathcal{H}$ satisfying
		\begin{equation}
		\pi(\mathcal{L}(\mathcal{N}))|\psi\rangle=\{0\};
		\end{equation}
		\item There exists a *-representation $\pi^{\prime}:\mathcal{U}(1) \rightarrow \mathcal{B}(\mathcal{H})$ and a tracial state $|\psi\rangle \in \mathcal{H}$ satisfying
		\begin{equation}
		\pi^{\prime}(\mathcal{J}(\operatorname{mir1}))|\psi\rangle=\{0\} ;
		\end{equation}
		\item There exists a *-representation $\pi^{\prime\prime}:\mathcal{U}(2) \rightarrow \mathcal{B}(\mathcal{H})$ and a tracial state $|\phi\rangle \in \mathcal{H}$ satisfying
		\begin{equation}
		\pi^{\prime\prime}(\mathcal{J}(\operatorname{mir2}))|\phi\rangle=\{0\};
		\end{equation}
		\item
		There exists a *-representation $\pi_{0}^{\prime}$ of $\mathcal{U}(1)$ mapping into a tracial von Neumann algebra $\mathcal{W} \subseteq \mathcal{B}(\mathcal{H})$ satisfying
		\begin{equation}
		\pi_{0}^{\prime}(\mathcal{J}(\operatorname{mir1}))=\{0\};
		\end{equation}
		\item
		There exists a *-representation $\pi_{0}^{\prime\prime}$ of $\mathcal{U}(2)$ mapping into a tracial von Neumann algebra $\mathcal{W}^{\prime} \subseteq \mathcal{B}(\mathcal{H})$ satisfying
		\begin{equation}
		\pi_{0}^{\prime\prime}(\mathcal{J}(\operatorname{mir2}))=\{0\}.
		\end{equation}
	\end{enumerate}
\end{thm}

To prove our main theorem, we introduce several  lemmas.
\begin{lem}\label{lem3.1}
	For every $ x\in X $ and $ a\in A $, we have $ e(1)^{x}_{a}-g^{\xi(x)}_{x,a}\in\mathcal{L}(\mathcal{N}) $. Similarly, for every $ y\in Y $ and $ b\in B $, we have $ e(2)^{y}_{b}-f^{\eta(y)}_{y,b}\in\mathcal{L}(\mathcal{N}) $.
\end{lem}
\begin{proof}
	Firstly, by $\sum_{a^{\prime}}e(1)^{x}_{a^{\prime}}=1$,  we have
	\begin{equation*}
	\begin{aligned}
	&\quad e(1)^{x}_{a}-g^{\xi(x)}_{x,a}&\\&=e(1)^{x}_{a}-\left(\sum_{a^{\prime}\in A}e(1)^{x}_{a^{\prime}}
	\right)g^{\xi(x)}_{x,a}
	%&\quad\left(\text{by}~\sum_{a^{\prime}}e(1)^{x}_{a^{\prime}}=1\right)&\\
	=e(1)^{x}_{a}-e(1)^{x}_{a}\cdot g^{\xi(x)}_{x,a}&\\
	&\quad+\left(\sum_{a^{\prime}\neq a}e(1)^{x}_{a^{\prime}}\right)\cdot\left(\sum_{b\in B,\lambda(x,\xi(x),a,b)=1}e(2)^{\xi(x)}_{b}\right)&\\
	&=e(1)^{x}_{a}\cdot\left(1-\sum_{b\in B,\lambda(x,\xi(x),a,b)=1}e(2)^{\xi(x)}_{b}\right)&\\
	&\quad +\sum_{a^{\prime}\neq a}
	\sum_{b\in B,\lambda(x,\xi(x),a,b)=1}
	e(1)^{x}_{a^{\prime}}e(2)^{\xi(x)}_{b}.&
	\end{aligned}
	\end{equation*}
	Notice that
	$$
	\begin{aligned}
	&\quad 1-\sum_{b\in B,\lambda(x,\xi(x),a,b)=1}e(2)^{\xi(x)}_{b}
	=\sum_{b\in B,\lambda(x,\xi(x),a,b)=0}e(2)^{\xi(x)}_{b}.
	\end{aligned}
	$$
	By the definition of $\mathcal{L}(\mathcal{N})$, we have
	\begin{eqnarray}
	&\qquad e(1)^{x}_{a}\left(1-\sum_{b\in B,\lambda(x,\xi(x),a,b)=1}e(2)^{\xi(x)}_{b}\right)\nonumber\\
&=\sum_{b\in B,\lambda(x,\xi(x),a,b)=0}e(1)^{x}_{a}e(2)^{\xi(x)}_{b}\in\mathcal{L}(\mathcal{N}).
	\end{eqnarray}
	On the other hand, it is known by the definition of mirror games that
	$ \lambda(x,\xi(x),a^{\prime},b)=0 $
	when $ a^{\prime}\neq a $ and $ \lambda(x,\xi(x),a,b)=1 $.  Hence we have
	$
	e(1)^{x}_{a^{\prime}}e(2)^{\xi(x)}_{b}\in\mathcal{N}
	$
	which implies
	\begin{equation}
	\sum_{a^{\prime}\neq a}
	\sum_{b\in B,\lambda(x,\xi(x),a,b)=1}
	e(1)^{x}_{a^{\prime}}e(2)^{\xi(x)}_{b}\in\mathcal{L}(\mathcal{N}).
	\end{equation}
	Therefore, we have
	\begin{equation}
	e(1)^{x}_{a}-g^{\xi(x)}_{x,a}\in\mathcal{L}(\mathcal{N}).
	\end{equation}
	
	Similarly, $ e(2)^{y}_{b}-f^{\eta(y)}_{y,b} $ can be rewritten as follows:
	\begin{align*}
	\quad e(2)^{y}_{b}-f^{\eta(y)}_{y,b}
	&=e(2)^{y}_{b}-f^{\eta(y)}_{y,b}\left(\sum_{b^{\prime}\in B}e(2)^{y}_{b^{\prime}}\right )\\
	&=e(2)^{y}_{b}-f^{\eta(y)}_{y,b}e(2)^{y}_{b}+\sum_{b^{\prime}\neq b}f^{\eta(y)}_{y,b}e(2)^{y}_{b^{\prime}}\\
	&=\left(1-\sum_{a\in A,\lambda(\eta(y),y,a,b)=1}e(1)^{\eta(y)}_{a}\right)
	\cdot e(2)^{y}_{b}\\
	&+\sum_{b^{\prime}\neq b}
	\sum_{a\in A,\lambda(\eta(y),y,a,b)=1}
	e(1)^{\eta(y)}_{a}e(2)^{y}_{b^{\prime}}.
	\end{align*}
	We still have
	\begin{eqnarray}
	&\quad\left(1-\sum_{a\in A,\lambda(\eta(y),y,a,b)=1}e(1)^{\eta(y)}_{a}\right)
	\cdot e(2)^{y}_{b}\nonumber\\
	&=\sum_{a\in A,\lambda(\eta(y),y,a,b)=0}e(1)^{\eta(y)}_{a}e(2)^{y}_{b}\in\mathcal{L}(\mathcal{N}),
	\end{eqnarray}
	and {by the definition of the mirror game}, we have
	\begin{equation}
	\sum_{b^{\prime}\neq b}
	\sum_{a\in A,\lambda(\eta(y),y,a,b)=1}
	e(1)^{\eta(y)}_{a}e(2)^{y}_{b^{\prime}}\in\mathcal{L}(\mathcal{N}).
	\end{equation}
	Therefore, we have
	\begin{equation}
	e(2)^{y}_{b}-f^{\eta(y)}_{y,b}\in\mathcal{L}(\mathcal{N}).
	\end{equation}
\end{proof}

\begin{lem}\label{lem3.2}
	We have the following inclusion relations:
	\begin{eqnarray}
	\{e(1)^{x}_{a}f^{\eta(y)}_{y,b}\mid\lambda(x,y,a,b)=0\}\subseteq\mathcal{L}(\mathcal{N}); \\
	\{e(2)^{y}_{b}g^{\xi(x)}_{x,a}\mid\lambda(x,y,a,b)=0\}\subseteq\mathcal{L}(\mathcal{N}).
	\end{eqnarray}
\end{lem}
\begin{proof}
	Notice that
	\begin{align*}
	e(1)^{x}_{a}f^{\eta(y)}_{y,b}=e(1)^{x}_{a}e(2)^{y}_{b}-e(1)^{x}_{a}\left(e(2)^{y}_{b}-f^{\eta(y)}_{y,b}\right)
	\end{align*}
	As $ \lambda(x,y,a,b)=0 $, we know $ e(1)^{x}_{a}e(2)^{y}_{b}\in\mathcal{N}\subseteq\mathcal{L}(\mathcal{N}) $. By Lemma \ref{lem3.1},  we have $ e(2)^{y}_{b}-f^{\eta(y)}_{y,b}\in\mathcal{L}(\mathcal{N}) $. Therefore, we have    \begin{equation}
	e(1)^{x}_{a}f^{\eta(y)}_{y,b}\in\mathcal{L}(\mathcal{N}).
	\end{equation}
	
	For $ e(2)^{y}_{b}g^{\xi(x)}_{x,a} $, we have
	$$
	\begin{aligned}
	e(2)^{y}_{b}g^{\xi(x)}_{x,a}&=e(2)^{y}_{b}e(1)^{x}_{a}-e(2)^{y}_{b}\left(e(1)^{x}_{a}-g^{\xi(x)}_{x,a}\right)\\
	&=e(1)^{x}_{a}e(2)^{y}_{b}-e(2)^{y}_{b}\left(e(1)^{x}_{a}-g^{\xi(x)}_{x,a}\right)\\
	&(\text{as $ e(1)^{x}_{a}$ always commutes with $ e(2)^{y}_{b} $})
	\end{aligned}
	$$
	We still have $ e(1)^{x}_{a}e(2)^{y}_{b}\in\mathcal{N}\subseteq\mathcal{L}(\mathcal{N}) $ by $ \lambda(x,y,a,b)=0 $,
	and $ e(1)^{x}_{a}-g^{\xi(x)}_{x,a}\in\mathcal{L}(\mathcal{N}) $ by Lemma \ref{lem3.1}.
	Then we have
	\begin{equation}
	e(2)^{y}_{b}g^{\xi(x)}_{x,a}\in\mathcal{L}(\mathcal{N}).
	\end{equation}
\end{proof}

\begin{lem}\label{lem3.3} We have
	$ \mathcal{J}(\operatorname{mir1})\subseteq\mathcal{L}(\mathcal{N}) $ and $ \mathcal{J}(\operatorname{mir2})\subseteq\mathcal{L}(\mathcal{N}) $.
\end{lem}
\begin{proof}
	Firstly let us  consider a monomial
\[ w(e(1))=e(1)^{x_1}_{a_1}\cdots e(1)^{x_t}_{a_t}\in \mathcal{U}(1),\] we have:
	\begin{align*}
	&\quad e(1)^{x_1}_{a_1}\cdots e(1)^{x_t}_{a_t}-g^{\xi(x_t)}_{x_t,a_t}e(1)^{x_1}_{a_1}\cdots e(1)^{x_{t-1}}_{a_{t-1}}\\
	&=
	e(1)^{x_1}_{a_1}\cdots e(1)^{x_t}_{a_t}-e(1)^{x_1}_{a_1}\cdots e(1)^{x_{t-1}}_{a_{t-1}}g^{\xi(x_t)}_{x_t,a_t}\\
	&=e(1)^{x_1}_{a_1}\cdots e(1)^{x_{t-1}}_{a_{t-1}}\left(e(1)^{x_t}_{a_t}-g^{\xi(x_t)}_{x_t,a_t}\right)\in\mathcal{L}(\mathcal{N}).
	\end{align*}
	Then we have
	\begin{equation}\label{eq1}
	\begin{aligned}
	&\qquad e(1)^{x_1}_{a_1}\cdots e(1)^{x_{t-1}}_{a_{t-1}} e(1)^{x_t}_{a_t}-g^{\xi(x_t)}_{x_t,a_t}g^{\xi(x_{t-1})}_{x_{t-1},a_{t-1}}\cdots g^{\xi(x_1)}_{x_1,a_1}\\
	&=e(1)^{x_1}_{a_1}\cdots e(1)^{x_t}_{a_t}-g^{\xi(x_t)}_{x_t,a_t}e(1)^{x_1}_{a_1}\cdots e(1)^{x_{t-1}}_{a_{t-1}}\\
	&+g^{\xi(x_t)}_{x_t,a_t}\left(e(1)^{x_1}_{a_1}\cdots e(1)^{x_{t-1}}_{a_{t-1}}
	-g^{\xi(x_{t-1})}_{x_{t-1},a_{t-1}}e(1)^{x_1}_{a_1}\cdots e(1)^{x_{t-2}}_{a_{t-2}}\right)\\
	&+\cdots+g^{\xi(x_t)}_{x_t,a_t}\cdots g^{\xi(x_2)}_{x_2,a_2}\left(e(1)^{x_1}_{a_1}-g^{\xi(x_1)}_{x_1,a_1}\right)
	\in\mathcal{L}(\mathcal{N}).
	\end{aligned}
	\end{equation}
	It's known  that $ g^{\xi(x_t)}_{x_t,a_t}\cdots g^{\xi(x_1)}_{x_1,a_1}=w^{*}(g) $.
	Then equation (\ref{eq1})
	can be written as
	\begin{equation}\label{we1-w*ginLN}
	w(e(1))-w^{*}(g)\in\mathcal{L}(\mathcal{N}).
	\end{equation}
	
	Suppose
	$$
	p=\sum_{\tiny{\begin{array}{c}x,y,a,b,~\lambda(x,y,a,b)=0\\u,w\end{array}}} u(e(1))\cdot e(1)^{x}_{a}f^{\eta(y)}_{y,b}\cdot w(e(1))\in\mathcal{J}(\operatorname{mir1}),
	$$
	where $ u(e(1)),w(e(1)) $ are monomials in $ \mathcal{U}(1) $ and $ \lambda(x,y,a,b)=0 $, we compute:
	\begin{align*}
	p&=\sum\Big(u(e(1))\cdot e(1)^{x}_{a}f^{\eta(y)}_{y,b}\cdot \left(w(e(1))-w^{*}(g)\right)\\
	&\quad+u(e(1))\cdot e(1)^{x}_{a}f^{\eta(y)}_{y,b}w^{*}(g)\Big)\\
	&=\sum\Big(u(e(1))\cdot e(1)^{x}_{a}f^{\eta(y)}_{y,b}\cdot \left(w(e(1))-w^{*}(g)\right)\\
	&\quad+w^{*}(g)~u(e(1))\cdot e(1)^{x}_{a}f^{\eta(y)}_{y,b}\Big).
	\end{align*}
	The second $ "=" $ is true  because $ w^{*}(g) $ is a polynomial in $ \mathcal{U}(2) $,
	which commutes with all of elements in $ \mathcal{U}(1) $. By the equation $ (\ref{we1-w*ginLN})$, we know that
	\begin{equation}	
	u(e(1))\cdot e(1)^{x}_{a}f^{\eta(y)}_{y,b}\cdot \left(w(e(1))-w^{*}(g)\right)\in\mathcal{L}(\mathcal{N}),
	\end{equation}
	and from Lemma \ref{lem3.2} we know
	\begin{equation}
	w^{*}(g)u(e(1))\cdot e(1)^{x}_{a}f^{\eta(y)}_{y,b}\in\mathcal{L}(\mathcal{N}).
	\end{equation}
	Then we conclude that every $ p\in\mathcal{J}(\operatorname{mir1}) $ satisfies $ p\in\mathcal{L}(\mathcal{N}) $, which means $ \mathcal{J}(\operatorname{mir1})\subseteq\mathcal{L}(\mathcal{N}) $.
	
	Similarly,   we have
	\begin{align}
	&\quad e(2)^{y_1}_{b_1}\cdots
	e(2)^{y_t}_{b_t}-f^{\eta(y_t)}_{y_t,b_t}e(2)^{y_1}_{b_1}\cdots e(2)^{y_{t-1}}_{y_{t-1}}\nonumber
	\\&=
	e(2)^{y_1}_{b_1}\cdots e(2)^{y_t}_{b_t}-e(2)^{y_1}_{b_1}\cdots e(2)^{y_{t-1}}_{b_{t-1}}f^{\eta(y_t)}_{y_t,b_t}\nonumber\\
	&=e(2)^{y_1}_{b_1}\cdots e(2)^{y_{t-1}}_{b_{t-1}}(e(2)^{y_t}_{b_t}-f^{\eta(y_t)}_{y_t,b_t})\in\mathcal{L}(\mathcal{N}).	
	\end{align}
	It is also true that
	\begin{eqnarray}
	&\quad w(e(2))-w^{*}(f)
	=e(2)^{y_1}_{b_1}\cdots e(2)^{y_{t-1}}_{b_{t-1}} e(2)^{y_t}_{b_t}\nonumber\\
	&\quad-f^{\eta(y_t)}_{y_t,b_t}f^{\eta(y_{t-1})}_{y_{t-1},b_{t-1}}\cdots f^{\eta(y_1)}_{y_1,y_1}
	\in\mathcal{L}(\mathcal{N}).
	\end{eqnarray}
	Then for
	$$ q=\sum u(e(2))\cdot e(2)^{y}_{b}g^{\xi(x)}_{x,a}\cdot w(e(2))\in\mathcal{J}(\operatorname{mir2}), $$
	we also  have
	\begin{eqnarray}
	q&=\sum\Big(u(e(2))\cdot e(2)^{y}_{b}g^{\xi(x)}_{x,a}\cdot \left(w(e(2))-w^{*}(f)\right) \nonumber\\
&\quad+u(e(2))\cdot e(2)^{y}_{b}g^{\xi(x)}_{x,a}w^{*}(f)\Big) \nonumber\\
	&=\sum\Big(u(e(2))\cdot e(2)^{y}_{b}g^{\xi(x)}_{x,a}\cdot \left(w(e(2))-w^{*}(f)\right)\nonumber\\
&\quad+w^{*}(f)u(e(2))\cdot e(2)^{y}_{b}g^{\xi(x)}_{x,a}\Big)\in\mathcal{L}(\mathcal{N}).
	\end{eqnarray}
	Therefore we have $ \mathcal{J}(\operatorname{mir2})\subseteq\mathcal{L}(\mathcal{N}) $.
\end{proof}

\begin{lem}\label{lem3.4}
	Let $ (\pi,|\psi\rangle) $ be a perfect commuting operator strategy of a regular mirror game $ \mathcal{G} $, then $ |\psi\rangle $ is a tracial state on both $ \pi(\mathcal{U}(1)) $ and $ \pi(\mathcal{U}(2)) $.
\end{lem}
\begin{proof}
	For the case $ \pi(\mathcal{U}(1)) $, it suffices to show that for any different $ e(1)^{x_1}_{a_1} $ and $ e(1)^{x_2}_{a_2} $, we have
	$$
	\langle\psi|\pi(e(1)^{x_1}_{a_1})\pi(e(1)^{x_2}_{a_2})|\psi\rangle=
	\langle\psi|\pi(e(1)^{x_2}_{a_2})\pi(e(1)^{x_1}_{a_1})|\psi\rangle,
	$$
	and we can complete the proof by using inductions on the length of monomials and linearity.
	
	In fact, since $ (\pi,|\psi\rangle) $ is a perfect commuting operator strategy of $ \mathcal{G} $, we have
	$ \pi(\mathcal{L}(\mathcal{N}))|\psi\rangle=\{0\} $ according to Definition \ref{def2.3} and Proposition \ref{thm2.1}. Lemma \ref{lem3.1} tells us that every $ e(1)^{x}_{a}-g^{\xi(x)}_{x,a}\in\mathcal{L}(\mathcal{N}) $, so we have
	$$
	\pi(e(1)^{x_1}_{a_1}-g^{\xi(x_1)}_{x_1,a_1})|\psi\rangle=0,~~~\text{and}~~~\pi(e(1)^{x_2}_{a_2}-g^{\xi(x_2)}_{x_2,a_2})|\psi\rangle=0.
	$$
	%Therefore, we have
	%	$$ \langle\psi|\pi(e(1)^{x_1}_{a_1})=\langle\psi|\pi(g^{\xi(x_1)}_{x_1,a_1}),~~~\text{and}~~~
	%\pi(e(1)^{x_2}_{a_2})|\psi\rangle=\pi(g^{\xi(x_2)}_{x_2,a_2})|\psi\rangle.
	%$$
	Therefore, we have
	$$
	\begin{array}{lll}
	&\quad\langle\psi|\pi(e(1)^{x_1}_{a_1})\pi(e(1)^{x_2}_{a_2})|\psi\rangle&\\&=
	\langle\psi|\pi(e(1)^{x_1}_{a_1})\pi(g^{\xi(x_2)}_{x_2,a_2})|\psi\rangle\\&=\langle\psi|\pi\left(e(1)^{x_1}_{a_1}g^{\xi(x_2)}_{x_2,a_2}\right)|\psi\rangle&(\pi~\text{is a representation})\\
	&=\langle\psi|\pi\left(g^{\xi(x_2)}_{x_2,a_2}e(1)^{x_1}_{a_1}\right)|\psi\rangle&(g^{\xi(x_2)}_{x_2,a_2}~\text{commutes with}~e(1)^{x_1}_{a_1})\\
	&=\langle\psi|\pi(g^{\xi(x_2)}_{x_2,a_2})\pi(e(1)^{x_1}_{a_1})|\psi\rangle\\
	&=\langle\psi|\pi(e(1)^{x_2}_{a_2})\pi(e(1)^{x_1}_{a_1})|\psi\rangle.
	\end{array}
	$$
	This shows that  $ |\psi\rangle $ is a tracial state on $ \pi(\mathcal{U}(1)) $.
	
	Similarly,~using $ e(2)^{y}_{b}-f^{\eta(y)}_{y,b}\in\mathcal{L}(\mathcal{N}) $, we can prove that $ |\psi\rangle $ is a tracial state on $ \pi(\mathcal{U}(2)) $.
\end{proof}

Now we can prove our main theorem.
\begin{proof}[Proof of Theorem\ref{thm3.2}]
	We show that  $(1) \iff(2)\iff(4)$ and $(1) \iff(3)\iff(5)$.
	
	Firstly, (1) is equivalent to  the existence of a perfect commuting operator strategy by the definition of the determining set.\vspace{1em}
	
	$(1)\Longrightarrow(2)$:
	Suppose $ (\pi,|\psi\rangle) $ is a pair that satisfies the conditions in (1), and we let $ \pi^{\prime} $ be the restriction of $ \pi $ to $ \mathcal{U}(1) $. It is obvious 
	$$
	\pi^{\prime}(\mathcal{J}(\operatorname{mir1}))|\psi\rangle=\pi(\mathcal{J}(\operatorname{mir1}))|\psi\rangle\subseteq\pi(\mathcal{L}(\mathcal{N}))|\psi\rangle=\{0\},
	$$
	where the first $ "=" $ comes from the restriction, the $ "\subseteq" $ is derived  from Lemma \ref{lem3.3}, and the second $ "=" $ is derived  from Proposition \ref{thm2.1}. By Lemma \ref{lem3.4}, We know $ |\psi\rangle $ is a tracial state.  Then $(1)\Longrightarrow(2)$ has been proved.\vspace{1em}
	
	$(2)\Longrightarrow(1)$:
	Using $(\pi^{\prime}, \psi)$,  we  define the following positive linear functional
	$$
	\ell^{\prime}: \mathcal{U}(1) \rightarrow \mathbb{C}, \quad h \mapsto \langle\psi|\pi^{\prime}(h)|\psi\rangle.
	$$
	Since $ |\psi\rangle $ is a tracial state, we know $ \ell^{\prime} $ is tracial. Next extend $\ell^{\prime}$ to a linear functional $\ell$ on $\mathcal{U}$ by mapping a monomial
	$$
	\ell:~w(e(1)) u(e(2)) \mapsto \ell^{\prime}\left(w(e(1)) u^{*}(f)\right),
	$$
	where $$ u^{*}(f)=f^{\eta(y_t)}_{y_t,b_t}f^{\eta(y_{t-1})}_{y_{t-1},b_{t-1}}\cdots f^{\eta(y_1)}_{y_1,b_1} $$
	if $ u(e(2))=e(2)^{y_1}_{b_1}\cdots e(2)^{y_{t-1}}_{b_{t-1}} e(2)^{y_t}_{b_t} $.
	We  show that  $ \ell $ is well-defined. It is sufficient  to show that  $ \ell $ is well defined on $ \mathbb{C} $.  Notice that the regularity ensures that:
	$$
	\sum_{b\in B}f^{\eta(y)}_{y,b}=1, ~\forall y\in Y.~~~~~~~~~~~~(\text{by Lemma \ref{lem2.1}})
	$$
	Then we have
	$$
	\ell\left(\sum_{b\in B}e(2)^{y}_{b}\right)=\ell^{\prime}\left(\sum_{b\in B}f^{\eta(y)}_{y,b}\right)=\langle\psi|\psi\rangle=1,~\forall y\in Y.
	$$
	On the other hand, 
 \[\ell\left(\sum_{b\in B}e(2)^{y}_{b}\right)=\ell(1)=\ell^{\prime}(1)=1.\]
  Therefore,  $ \ell $ is really well defined.
	
	Next we show that $ \ell $ can distinguish $ -1 $ and $ \operatorname{SOS}_{\mathcal{U}}+\mathcal{L}(\mathcal{N})+\mathcal{L}(\mathcal{N})^{*} $. The motivation of this proof is similar to the proof of Theorem 8.3 in \cite{watts2021noncommutative}.
	
	Since $ |\psi\rangle $ is a tracial state, we know that
	$\ell$ is symmetric in a sense that
	$\ell\left(h^*\right)=\ell(h)^*$ for all $h \in \mathcal{U}(1)$.
	To check that $\ell$ is positive, let $h=\sum_{i, j} \beta_{i j} w_i(e(1)) u_j(e(2)) \in \mathcal{U}$, then we have
	$$
	h^{*}h=\sum_{i, j} \sum_{k, s} \beta_{i j}^* \beta_{k s}\cdot w_i^*(e(1)) w_k(e(1)) u_j^*(e(2)) u_s(e(2)),
	$$
	whence
	\begin{equation}\label{eq2}
	\ell\left(h^* h\right)=\sum_{i, j} \sum_{k, s} \beta_{i j}^* \beta_{k s}\cdot\ell^{\prime}\left(w_i^*(e(1)) w_k(e(1)) u_s^*(f) u_j(f)\right) .
	\end{equation}
	Set
	$$
	\check{h}=\sum_{i, j} \beta_{i j} w_i(e(1)) u_j^*(f) \in \mathcal{U}(1)
	$$
	Then we have 
	$$
	\check{h}^* \check{h}=\sum_{i, j} \sum_{k, s} \beta_{i j}^* \beta_{k s} u_j(f) w_i^*(e(1)) w_k(e(1)) u_s^*(f),
	$$
	\begin{equation}\label{eq3}
	\ell^{\prime}\left(\check{h}^* \check{h}\right)=\sum_{i, j} \sum_{k, s} \beta_{i j}^* \beta_{k s} \ell^{\prime}\left(u_j(f) w_i^*(e(1)) w_k(e(1)) u_s^*(f)\right) .
	\end{equation}
	Since $\ell^{\prime}$ is tracial, we have
	$$
	\begin{aligned}
	&\quad\ell^{\prime}\left(w_i^*(e(1)) w_k(e(1)) u_s^*(f) u_j(f)\right)\\&=\ell^{\prime}\left(u_j(f) w_i^*(e(1)) w_k(e(1)) u_s^*(f)\right) .
	\end{aligned}
	$$
	This implies that the values in Equation (\ref{eq2}) and Equation (\ref{eq3})  are the same. Therefore, we have
	$$
	\ell\left(h^* h\right)=\ell^{\prime}\left(\check{h}^* \check{h}\right) \geq 0,
	$$
	which implies
	$ \ell(\operatorname{SOS}_{\mathcal{U}})\geq 0.$
	
	It remains to show that $\ell(\mathcal{L}(\mathcal{N}))=\{0\}$. Elements in $\mathcal{L}(\mathcal{N})$ are linear combinations of monomials of the form
	\begin{equation}\label{eq4}
	w(e(1)) u(e(2)) e(1)^{x}_{a} e(2)^{y}_{b}=w(e(1)) e(1)^{x}_{a} u(e(2)) e(2)^{y}_{b}
	\end{equation}
	with $\lambda(x,y,a,b)=0$. Applying $\ell$ to Equation  (\ref{eq4}) gives that
	$$
	\ell\left(w(e(1)) e(1)^{x}_{a} u(e(2)) e(2)^{y}_{b}\right) =\ell^{\prime}\left(w(e(1)) e(1)^{x}_{a} f^{\eta(y)}_{y,b} u^*(f)\right).
	$$
	But $e(1)^{x}_{a} f^{\eta(y)}_{y,b}\in\mathcal{J}(\operatorname{mir1})$, whence 
\[w(e(1)) e(1)^{x}_{a} f^{\eta(y)}_{y,b} u^*(f)\in \mathcal{J}(\operatorname{mir1}).\]
 Hence we have
	$$
	\ell^{\prime}\left(w(e(1)) e(1)^{x}_{a} f^{\eta(y)}_{y,b} u^*(f)\right)=0.
	$$
	We have proved that $\ell(-1)=-1$ and $$\ell\left(\operatorname{SOS}_{\mathscr{U}}+\mathfrak{L}(\mathcal{N})+\mathfrak{L}(\mathcal{N})^*\right) \subseteq \mathbb{R}_{\geq 0},$$ whence $-1 \notin \operatorname{SOS}_{\mathscr{U}}+\mathfrak{L}(\mathcal{N})+\mathfrak{L}(\mathcal{N})^*$, which implies (1) by  Theorem 4.3 in \cite{watts2021noncommutative}.\vspace{1em}
	
	$(2)\Longrightarrow(4)$: Now we have the pair $ (\pi^{\prime},|\psi\rangle) $, where $ \pi^{\prime}:\mathcal{U}(1)\rightarrow\mathcal{B}(\mathcal{H}) $ is a *-representation and $ |\psi\rangle $ is a tracial state. Now we construct the von Neumann algebra $ \mathcal{W} $ and $ \pi_{0}^{\prime}:\mathcal{U}(1) \rightarrow\mathcal{W}$.
	
	We denote the completion of 
	\(\{\pi^{\prime}(\mathcal{U}(1))|\psi\rangle\}\subseteq\mathcal{H}\)
	as $ \check{\mathcal{H}} $, 
	and it's obvious that $ \check{\mathcal{H}} $ is a closed subspace of $ \mathcal{H} $. 
	Then we have
	$$\pi^{\prime}(\mathcal{U}(1))\check{\mathcal{H}}\subseteq\check{\mathcal{H}},$$ and  $ \pi^{\prime} $ induces a *-representation $ \check{\pi}^{\prime}:~\mathcal{U}(1)\rightarrow\mathcal{B}(\check{\mathcal{H}}) $ naturally.
	We let $ \mathcal{W}=\mathcal{B}(\check{\mathcal{H}}) $ and $ \pi_{0}^{\prime}=\check{\pi}^{\prime} $ as what we desire. Next we'll prove that $ \mathcal{W} $ and $ \pi_{0}^{\prime} $ satisfy the requirement of the item (4).
	
	Firstly notice that $ \mathcal{B}(\check{\mathcal{H}}) $ is a von Neumann algebra because it is closed in the weak operator topology. 
	
	Secondly, since $ |\psi\rangle $ is a tracial state on $ \mathcal{B}(\mathcal{H}) $, $ |\psi\rangle $ is also a tracial state on $ B(\check{\mathcal{H}})$. 
	Thus $ \tau:\mathcal{B}(\check{\mathcal{H}})\rightarrow\mathbb{C}, a\mapsto\langle\psi|a|\psi\rangle $ is a tracial linear functional on $ \mathcal{B}(\check{\mathcal{H}}) $, and
	$ (\mathcal{B}(\check{\mathcal{H}}),\tau) $ is a tracial von Neumann algebra.
	
	Lastly, to show $ \check{\pi}^{\prime}(\mathcal{J}(\operatorname{mir1}))=\{0\} $, it suffices to show the following claim:\\
	For any $u\in\mathcal{U}(1)$ and $ |\phi\rangle=\pi^{\prime}(u)|\psi\rangle\in\check{\mathcal{H}} $, we have
	\[\check{\pi}^{\prime}(\mathcal{J}(\operatorname{mir1}))|\phi\rangle=\{0\}.\]
	
	We have
	$$
	\begin{array}{lll}
	&\quad\check{\pi}^{\prime}(\mathcal{J}(\operatorname{mir1}))|\phi\rangle&\\
	&=\pi^{\prime}(\mathcal{J}(\operatorname{mir1})u)|\psi\rangle&
	(\text{the definition of }\phi)\\
	&=\pi^{\prime}(\mathcal{J}(\operatorname{mir1}))|\psi\rangle&(\mathcal{J}(\operatorname{mir1})~\text{is a two-sided ideal})\\
	&=\{0\}&(\text{by item (2)}).
	\end{array}
	$$
	%i.e. $ \check{\pi}^{\prime}(\mathcal{J}(\operatorname{mir1}))=\{0\} $.
	%Then we omit the proof. \vspace{1em}
	
	$(4)\Longrightarrow(2)$: We start with the tracial von Neumann algebra $\mathcal{W}$ with trace $\tau$ defined  in (4) and perform a Gelfand-Naimark-Segal (GNS) construction \cite{kadison1983fundamentals}. There is a Hilbert space $\mathcal{K}$, a unit vector $\rho\in\mathcal{K}$, and a *-representation $\pi_{1}^{\prime}: \mathcal{W} \rightarrow \mathcal{B}(\mathcal{K})$ such that
	$$
	\tau(a)=\left\langle\pi_{1}^{\prime}(a)\rho, \rho\right\rangle, \quad a \in \mathcal{W} .
	$$
	Since $\tau$ is a trace, $\rho$ is a tracial state for $\pi_{1}^{\prime}(\mathcal{W})$. Then the *-representation $\pi_{1}^{\prime}\circ\pi_{0}^{\prime}: \mathcal{U}(1) \rightarrow \mathcal{B}(\mathcal{K})$ together with $\rho\in\mathcal{K}$ satisfy (2).
	\vspace{1em}
	
	$(1)\Longrightarrow(3)$ is similar to $(1)\Longrightarrow(2)$, but using  another side of Lemma \ref{lem3.3} and Lemma \ref{lem3.4}.\vspace{1em}
	
	$(3)\Longrightarrow(1)$ is similar to $(2)\Longrightarrow(1)$.  Here, we extend $ \ell^{\prime}:\mathcal{U}(2)\rightarrow\mathbb{C} $ to a linear functional $\ell_1$ from  algebra $\mathcal{U}$ to $\mathbb{C} $: %the linear functional $\ell_1$ is defined as %there exists a difference: 
	$$
	\ell_1:~u(e(2))w(e(1))\mapsto \ell^{\prime}\left(u(e(2))w^{*}(g)\right).
	$$
	As $ \sum_{a\in A}g^{\xi(x)}_{x,a}=1,~\forall x\in X $, we know that  $ \ell_1 $ is well defined.
	The proof of $ \ell_1(\operatorname{SOS}_{\mathcal{U}(1)})\geq 0 $ is similar to the corresponding part in $(2)\Longrightarrow(1)$.
	
	For the proof of $ \ell_1(\mathcal{L}(\mathcal{N}))=\{0\} $, since elements in $ \mathcal{U}(1) $ commute with those in $ \mathcal{U}(2) $, elements in $ \mathcal{N} $ can also be written as $ e(2)^{y}_{b}e(1)^{x}_{a} $.  Then elements in $\mathcal{L}(\mathcal{N})$ are linear combinations of monomials of the form
	\begin{equation}\label{eq5}
	w(e(1)) u(e(2)) e(2)^{y}_{b} e(1)^{x}_{a}=u(e(2)) e(2)^{y}_{b} w(e(1)) e(1)^{x}_{a}
	\end{equation}
	with $\lambda(x,y,a,b)=0$. Applying the new $\ell_1$ to  (\ref{eq5})  gives that
	$$
	\ell_1\left(u(e(2)) e(2)^{y}_{b} w(e(1)) e(1)^{x}_{a}\right) =\ell^{\prime}\left(u(e(2)) e(2)^{y}_{b}g^{\xi(x)}_{x,a}w^*(g)\right).
	$$
	But $e(2)^{y}_{b} g^{\xi(x)}_{x,a}\in\mathcal{J}(\operatorname{mir2})$, whence $u(e(2)) e(2)^{y}_{b}g^{\xi(x)}_{x,a}w^*(g)\in\mathcal{J}(\operatorname{mir2})$. Therefore we derive that
	$$
	\ell^{\prime}\left(u(e(2)) e(2)^{y}_{b}g^{\xi(x)}_{x,a}w^*(g)\right)=0.
	$$
	i.e. we get $ \ell_1(\mathcal{L}(\mathcal{N}))=\{0\} $.
	Using  Theorem 4.3 in \cite{watts2021noncommutative}, we can show $(3)\Longrightarrow(1)$.\vspace{1em}
	
	Finally,  the proofs of $(3)\Longrightarrow(5)$  and   $(5)\Longrightarrow(3)$ are
	similar to the proofs of $(2)\Longrightarrow(4)$ and $(4)\Longrightarrow(2)$.
\end{proof}
\newtheorem{rmk}{Remark}
\begin{rmk}
	Theorem \ref{thm3.2} may not hold if a mirror game is not regular. For example, let the scoring function $ \lambda=0 $ for all questions and answers. It is not a regular mirror game since
	$ \mathcal{J}(\operatorname{mir1})=\mathcal{J}(\operatorname{mir2})=\{0\}$. Items (2),(4) in Theorem \ref{thm3.2}  always hold.  However,  we can easily verify that  $ \mathcal{G} $ can't have a perfect commuting operator strategy.  Therefore, Theorem \ref{thm3.2} is only true for regular mirror games.

\end{rmk}

In \cite{klep2008connes}, Klep and  Schweighofer show  that Connes’ embedding conjecture on von Neumann algebras is equivalent to the tracial version of the  Positivstellensatz.
See  \cite{BKP2016, klep2016constrained, klep2022optimization} for more recent progress in tracial optimizations.  It has been shown in  Theorem 8.7 \cite{watts2021noncommutative} that, given a *-algebra $ \mathcal{A} $ satisfying   the condition of Archimedean, i.e. for every $a \in \mathcal{A}$, there is  an $\varepsilon \in \mathbb{N}$ with $\varepsilon-a^* a \in \widetilde{\operatorname{SOS}}_{\mathcal{A}}$,
where 	
\begin{equation}
\widetilde{\operatorname{SOS }}_{\mathcal{A}}=\{a\in\mathcal{A}\mid\exists b\in\operatorname{SOS}_{\mathcal{A}}, ~ a-b ~{\text{is a sum of commutators}}\}.
\end{equation}
Then there exists a $*$-representation $\pi: \mathcal{A} \rightarrow \mathcal{B}(\mathcal{H})$ and a tracial state $0 \neq |\psi\rangle \in \mathcal{H}$ satisfying
\begin{equation}\label{eqtr1}
\pi(f) |\psi\rangle=0,~\text{for all}~ f \in \mathfrak{L},
\end{equation}
if and only if
there exists a $*$-representation $\pi: \mathcal{A} \rightarrow \mathcal{F}$ into a tracial von Neumann algebra $(\mathcal{F}, \tau)$ satisfying
\begin{equation}
\tau(\pi(f))=0,~\text{for all}~ f \in \mathfrak{L};
\end{equation}
which is also equivalent to
\begin{equation}\label{eqtr3}
-1 \notin \widetilde{\operatorname{SOS}}_\mathcal{A}+\mathfrak{L}+\mathfrak{L}^*,
\end{equation}

By cyclic unitary generators defined in \cite{watts2021noncommutative}, we can show that both $ \mathcal{U}(1) $ and $ \mathcal{U}(2) $ are group algebra. And by Example 4.4 of \cite{watts2021noncommutative} we know $ \mathcal{U}(1) $ and $ \mathcal{U}(2) $ are Archimedean.
Hence, we can combine the above equivalent condition $(\ref{eqtr1}),~(\ref{eqtr3})$ with item $ (2),(4) $ of our Theorem \ref{thm3.2}. Then we have the following corollary:
\begin{cor}\label{cor3.3}
	A regular mirror game with its universal game algebra $ \mathcal{U} $ and invalid determining set $ \mathcal{N} $ has a perfect commuting operator strategy if and only if
	$ -1\notin\widetilde{\operatorname{SOS}}_{\mathcal{U}(1)}+\mathcal{J}(\operatorname{mir1})+\mathcal{J}(\operatorname{mir1})^* $. For a special case, if a mirror game satisfies
	\begin{equation}
	-1\in\operatorname{SOS}_{\mathcal{U}(1)}+\mathcal{J}(\operatorname{mir1})+\mathcal{J}(\operatorname{mir1})^*,
	\end{equation}
	then it cannot have a perfect commuting strategy.  Similar results hold for $ \mathcal{J}(\operatorname{mir2}) $.
\end{cor}

Notice that $ \mathcal{J}(\operatorname{mir1}) $ is a two-sided ideal, so that $ \mathcal{J}(\operatorname{mir1})+\mathcal{J}(\operatorname{mir1})^* $ is still a two-sided ideal, generated by
\begin{equation}
\left\{e(1)^{x}_{a}f^{\eta(y)}_{y,b},~f^{\eta(y)}_{y,b}e(1)^{x}_{a}\mid\lambda(x,y,a,b)=0\right\}.
\end{equation}
Then we can use  the noncommutative Gr\"obner basis method to solve this ideal membership problem \cite{mora1986grobner, madlener1998string,levandovskyy2005non,xiu2012non}

\section{A Procedure  for Proving Nonexistence of  Perfect  Strategy}\label{sec4}

According to Corollary \ref{cor3.3},  we can prove that  a regular mirror game $ \mathcal{G} $ doesn't have a perfect commuting operator strategy using noncommutative  Gr\"obner basis and semidefinite programming.

The main steps of the procedure are listed as follows.
\begin{enumerate}[(1)]
	\item Let $ \mathbb{C}\langle e(1)\rangle $ be the free algebra generated by $ \{e(1)^{x}_{a}\mid x\in X,~a\in A\} $, and $ \Pi $ be the canonical projection from $ \mathbb{C}\langle e(1)\rangle $ onto $ \mathcal{U}(1)$. Then $ \Pi^{-1}(\mathcal{J}(\operatorname{mir1})) $ is a two-sided ideal in $ \mathbb{C}\langle e(1)\rangle $, generated by
	\begin{equation}\label{Pi-1Jmir1}
	\begin{aligned}
	&\quad\left\{e(1)^{x}_{a}f^{\eta(y)}_{y,b}\mid\lambda(x,y,a,b)=0\right\}\\&\cup\left\{(e(1)^{x}_{a})^2-e(1)^{x}_{a},~e(1)^{x}_{a_1}e(1)^{x}_{a_2},~\sum_{a\in A}e(1)^{x}_{a}-1\right\}.
	\end{aligned}
	\end{equation}
	Therefore $ \Pi^{-1}(\mathcal{J}(\operatorname{mir1}))+\Pi^{-1}(\mathcal{J}(\operatorname{mir1}))^* $ is a two-sided ideal generated by
	\begin{equation}\label{Pi-1Jmir1*}
	\begin{aligned}
	&\quad\left\{e(1)^{x}_{a}f^{\eta(y)}_{y,b},~f^{\eta(y)}_{y,b}e(1)^{x}_{a}\mid\lambda(x,y,a,b)=0\right\}\\
	&\cup\left\{(e(1)^{x}_{a})^2-e(1)^{x}_{a},~e(1)^{x}_{a_1}e(1)^{x}_{a_2},~\sum_{a\in A}e(1)^{x}_{a}-1\right\}.
	\end{aligned}
	\end{equation}
	\item  We  compute the noncommutative Gr\"obner basis  $ \operatorname{GB} $ \\ of $ \Pi^{-1}(\mathcal{J}(\operatorname{mir1}))+\Pi^{-1}(\mathcal{J}(\operatorname{mir1}))^*$. % denoted as $ \operatorname{GB} $.
	\begin{enumerate}
		\item   If $ 1\in\operatorname{GB}$, then we have
		\begin{equation}
		-1\in\operatorname{SOS}_{\mathbb{C}\langle e(1)\rangle}+\Pi^{-1}(\mathcal{J}(\operatorname{mir1}))+\Pi^{-1}(\mathcal{J}(\operatorname{mir1}))^*.
		\end{equation}
		Hence, we have
		\begin{equation}
		-1\in\operatorname{SOS}_{\mathcal{U}(1)}+\mathcal{J}(\operatorname{mir1})+\mathcal{J}(\operatorname{mir1})^*,
		\end{equation}
		which implies that the game can't have a perfect strategy.
		\item  Otherwise, we   check whether there exist polynomials \\$s_j \in\mathscr{U}(1)$ such that
		$$
		~1+\sum_{j=1}^{k}s_j^*s_j\in\Pi^{-1}(\mathcal{J}(\operatorname{mir1}))+\Pi^{-1}(\mathcal{J}(\operatorname{mir1}))^*.
		$$
		% This can be converted to  finding a feasible solution to a semidefinite programming problem.
		
		Let $ W_d $ be the column vector composed of monomials in $ \mathbb{C}\langle e(1)\rangle$ having a total degree less than or equal to $d $.  Using an SDP solver to test whether  there exists a positive semidefinite matrix $G$ such that
		\begin{equation}\label{sdp}
		~1+W_d^*GW_d\rightarrow_{\operatorname{GB}}0
		\end{equation}
		\begin{itemize}
			\item If (\ref{sdp}) has a solution $G$, then the mirror game can't have a perfect strategy.
			\item  Otherwise, set $d:=d+1$ and go back.
		\end{itemize}
	\end{enumerate}
\end{enumerate}

\begin{rmk}
Since a free algebra generated by two or more variables is non-Noetherian, Buchberger's procedure for computing  a non-commutative Gr\"obner basis   may not terminate \cite{mora1994introduction,xiu2012non}. Thus our procedure may not terminate in finite steps.

If the procedure stops at  some degree $d$, we can verify that the mirror game has no perfect 
commuting operator strategy. Otherwise, we do not know whether the mirror game has a perfect 
commuting operator strategy.  

In fact, according to \cite[Theorem 5.1]{lupini2020perfect}, 
   an imitation game $\mathcal{G}$ has a perfect commuting operator strategy if and only if a tracial state exists on the $C^*$-algebra $C^*(\mathcal{G})$. By \cite[Remark 2.21]{Fritz2018CuriousPO}, $C^*(\mathcal{G})$ is a free hypergraph C*-algebra, and there is no  algorithm to  determine  whether a free hypergraph C*-algebra has a tracial state \cite[Theorem 3.6]{Fritz2018CuriousPO}. Hence, there is no algorithm that terminates in finite steps to determine whether a mirror game (an imitation game)  has a perfect commuting operator strategy.

%According to \cite[Theorem 3.6]{Fritz2018CuriousPO}, there is no  algorithm to  determine  whether a free hypergraph C*-algebra has a tracial state and imitation game C*-algebra is also a free hypergraph C*-algebra \cite[Remark 2.21]{Fritz2018CuriousPO}. But \cite[Theorem 5.1]{lupini2020perfect} shows that an imitation game has a perfect commuting operator strategy if a tracial state exists on its game C*-algebra. Thus, no algorithm can determine whether an imitation game has a perfect commuting operator strategy. Since the mirror game is a subclass of the imitation game, the same conclusion holds on the mirror game.
\end{rmk}

\vspace{1em}
%has a perfect commuting operator strategy has no finite terminal algorithm. 

%Lastly we give an example.
%\begin{exa}
{Example} \ref{example1} (continued).
{\it 
	%Let us continue the computation in Example \ref{example1}.
	Let $ \mathbb{C}\langle e(1)\rangle $ be the free algebra generated by
	$ \{e(1)^{i}_{j}\mid(i,j)\in\{0,1\}^2\} $, and  $ \mathcal{U}(1)$ be the subalgebra of the universal game algebra $ \mathcal{U} $ generated by $\{e(1)^{i}_{j}\mid(i,j)\in\{0,1\}^2\}$. Then we have the natural projection $ \Pi:\mathbb{C}\langle e(1)\rangle\rightarrow \mathcal{U}(1) $.
	
	Notice that $ \mathcal{J}(\operatorname{mir1}) $ is *-closed. Hence,  we have 
 $$ \mathcal{J}(\operatorname{mir1})+\mathcal{J}(\operatorname{mir1})^*=\mathcal{J}(\operatorname{mir1}),$$
 and
	$$
	\begin{aligned}
	&\quad\Pi^{-1}(\mathcal{J}(\operatorname{mir1}))\\
	&=\{e(1)^{0}_{0},~e(1)^{0}_{1},~e(1)^{0}_{0}+e(1)^{0}_{1}-1,~e(1)^{1}_{0}+e(1)^{1}_{1}-1,\\
	&\qquad e(1)^{0}_{0}e(1)^{0}_{1},~e(1)^{0}_{1}e(1)^{0}_{0},~e(1)^{1}_{0}e(1)^{1}_{1},~e(1)^{1}_{1}e(1)^{1}_{0},\\
	&\qquad(e(1)^{0}_{0})^2-e(1)^{0}_{0},~(e(1)^{0}_{1})^2-e(1)^{0}_{1},\\
	&\qquad(e(1)^{1}_{0})^2-e(1)^{1}_{0},~(e(1)^{1}_{1})^2-e(1)^{1}_{1}\}
	\end{aligned}
	$$
	is a two-sided ideal in $ \mathscr{U}(1) $. It is evident that
	$$ -1\in\operatorname{SOS}_{\mathcal{U}(1)}+\mathcal{J}(\operatorname{mir1})\iff -1\in\operatorname{SOS}_{\mathbb{C}\langle e(1)\rangle}+\Pi^{-1}(\mathcal{J}(\operatorname{mir1})) $$
	Using the software NCAlgebra (\url{https://github.com/NCAlgebra/}), %we can get a noncommutative Gr\"oebner Basis of $ \Pi^{-1}(\mathcal{J}(\operatorname{mir1})) $.
	we can show that  $ 1 $ is in the Gr\"obner basis of $\Pi^{-1}(\mathcal{J}(\operatorname{mir1})) $
 i.e.,
	%\[\Pi^{-1}(\mathcal{J}(\operatorname{mir1}))=\mathscr{U}(1),\]
	$$ -1\in\Pi^{-1}(\mathcal{J}(\operatorname{mir1})).$$
	Therefore, this  game doesn't have a perfect commuting operator strategy.}
%\end{exa}
\bibliographystyle{ACM-Reference-Format}
\bibliography{document.bib}

\end{document}